\documentclass{amsart}
\usepackage{amssymb}
\usepackage{amsmath, amscd}
\usepackage{amsthm}
\usepackage{mathrsfs}
\usepackage{perpage}
\usepackage[all,cmtip]{xy}
\usepackage{setspace}
\usepackage{amsbsy}
\usepackage[foot]{amsaddr}
\usepackage{url}
\usepackage[T1]{fontenc}
\usepackage{verbatim}
\usepackage{mathrsfs}
\usepackage{ifthen}
\usepackage{blkarray}
\usepackage{multirow}
\usepackage{enumerate}
\usepackage{tikz-cd}
\usepackage{mathtools}
\usepackage[left=2cm,head=13pt,
top=1.5cm,right=2cm, bottom=1.5cm]{geometry}

\numberwithin{equation}{subsection}

\newtheorem{theorem}{Theorem}[subsection]
\newtheorem{lemma}[theorem]{Lemma}

\newtheorem{corollary}[theorem]{Corollary}
\newtheorem{definition}[theorem]{Definition}

\newtheorem{proposition}[theorem]{Proposition}

\newtheorem*{thm1}{Theorem 1}
\newtheorem*{thm2}{Theorem 2}

\newtheorem{cor}[theorem]{Corollary}
\newtheorem*{cor1}{Corollary 1}
\newtheorem*{cor2}{Corollary 2}

\theoremstyle{remark}

\newtheorem{rmk}[theorem]{Remark}

\title{Generalized $\mu$-ordinary Hasse invariants}

\newcommand{\GZip}{\mathop{\text{$G$-{\tt Zip}}}\nolimits}
\newcommand{\GjZ}{\mathop{\text{$G_j$-{\tt Zip}}}\nolimits}
\newcommand{\GtildeZ}{\mathop{\text{$\tilde{G}$-{\tt Zip}}}\nolimits}

\newcommand{\GoneZip}{\mathop{\text{$G_1$-{\tt Zip}}}\nolimits}
\newcommand{\GtwoZip}{\mathop{\text{$G_2$-{\tt Zip}}}\nolimits}

\DeclareMathOperator{\Stab}{Stab}
\DeclareMathOperator{\Gal}{Gal}
\newcommand{\ad}{\textnormal{ad}}

\newcounter{listcounter}
\newcounter{deflistcounter}
\newcounter{equivcounter}

\newskip{\itemsepamount}
\newskip{\topsepamount}
\newenvironment{assertionlist}{%
  \begin{list}
    {\upshape (\arabic{listcounter})}
    {\setlength{\leftmargin}{18pt}
     \setlength{\rightmargin}{0pt}
     \setlength{\itemindent}{0pt}
     \setlength{\labelsep}{5pt}
     \setlength{\labelwidth}{13pt}
     \setlength{\listparindent}{\parindent}
     \setlength{\parsep}{0pt}
     \setlength{\itemsep}{\itemsepamount}
     \setlength{\topsep}{\topsepamount}
     \usecounter{listcounter}}}
  {\end{list}}

\newenvironment{definitionlist}{%
  \begin{list}
    {\upshape (\alph{deflistcounter})}
    {\setlength{\leftmargin}{18pt}
     \setlength{\rightmargin}{0pt}
     \setlength{\itemindent}{0pt}
     \setlength{\labelsep}{5pt}
     \setlength{\labelwidth}{13pt}
     \setlength{\listparindent}{\parindent}
     \setlength{\parsep}{0pt}
     \setlength{\itemsep}{\itemsepamount}
     \setlength{\topsep}{\topsepamount}
     \usecounter{deflistcounter}}}
  {\end{list}}

\newenvironment{equivlist}{%
  \begin{list}
    {\upshape (\roman{equivcounter})}
    {\setlength{\leftmargin}{18pt}
     \setlength{\rightmargin}{0pt}
     \setlength{\itemindent}{0pt}
     \setlength{\labelsep}{5pt}
     \setlength{\labelwidth}{13pt}
     \setlength{\listparindent}{\parindent}
     \setlength{\parsep}{0pt}
     \setlength{\itemsep}{\itemsepamount}
     \setlength{\topsep}{\topsepamount}
     \usecounter{equivcounter}}}
  {\end{list}}

\DeclareMathOperator{\Coker}{Coker}

\newcommand{\Ecal}{{\mathcal E}}

\newcommand{\Gcal}{{\mathcal G}}

\newcommand{\Ocal}{{\mathcal O}}

\newcommand{\Scal}{{\mathcal S}}

\newcommand{\Vcal}{{\mathcal V}}

\newcommand{\Xcal}{{\mathcal X}}

\newcommand{\Zcal}{{\mathcal Z}}

\newcommand{\pfr}{{\mathfrak p}}

\renewcommand{\AA}{\mathbb{A}}

\newcommand{\CC}{\mathbb{C}}

\newcommand{\FF}{\mathbb{F}}
\newcommand{\GG}{\mathbb{G}}

\newcommand{\NN}{\mathbb{N}}

\newcommand{\QQ}{\mathbb{Q}}
\newcommand{\RR}{\mathbb{R}}
\renewcommand{\SS}{\mathbb{S}}

\newcommand{\ZZ}{\mathbb{Z}}

\DeclareMathOperator{\Pic}{Pic}
\DeclareMathOperator{\Ker}{Ker}

\newcommand{\Ascr}{{\mathscr A}}

\newcommand{\Lscr}{{\mathscr L}}

\newcommand{\Vscr}{{\mathscr V}}

\newcommand{\Ltilde}{\widetilde{L}}
\newcommand{\Mtilde}{\widetilde{M}}

\newcommand{\Ptilde}{\widetilde{P}}
\newcommand{\Qtilde}{\widetilde{Q}}

\DeclareMathOperator{\codim}{codim}

\newcommand{\Th}{{\rm Th.}}

\newcommand{\Rmk}{{\rm Rmk.}}

\newcommand{\Cor}{{\rm Cor.}}

\newcommand{\Lem}{{\rm Lem.}}

\newcommand{\Def}{{\rm Def.}}

\newcommand{\Prop}{{\rm Prop.}}

\newcommand{\loccit}{{\em loc.\ cit. }}
\newcommand{\loccitn}{{\em loc.\ cit.}}

\DeclareMathOperator{\Div}{div}

\DeclareMathOperator{\Hom}{Hom}
\DeclareMathOperator{\Res}{Res}

\DeclareMathOperator{\id}{id}
\DeclareMathOperator{\lcm}{lcm}

\author{Jean-Stefan Koskivirta, Torsten Wedhorn}
\date{\today}

\begin{document}

\begin{abstract}
We give a short proof for the existence of $\mu$-ordinary Hasse invariants for the good reduction special fiber of Shimura varieties of Hodge-type using the stack of $G$-zips introduced by Moonen-Wedhorn and Pink-Wedhorn-Ziegler. We give an explicit formula for the power of the Hodge bundle that admits a Hasse invariant. When $G$ is a Weil restriction of $G_1$, we relate the Ekedahl-Oort strata of $G$ and those of $G_1$.
\end{abstract}

\maketitle


\section*{Introduction}

The special fibers of integral models of Shimura varieties have nice stratifications by subvarieties which are equipped with interesting Hecke actions. They play an important role in many recent developments in the construction of $p$-adic modular forms and Galois representations. At primes of good reductions for Shimura variteies of Hodge type, all stratifications that are usually considered (Newton stratification, Ekedahl-Oort stratification, foliation by central leaves) have the same generic stratum (\cite{Wortmann-mu-ordinary}), named the $\mu$-ordinary stratum in \cite{Wedhorn-ordinariness-Shimura-varieties}.

In this article we construct a ``canonical'' Hasse invariants for the $\mu$-ordinary stratum, i.e., a section of some power of the Hodge line bundle over the special fiber whose non-vanishing locus is precisely the $\mu$-ordinary locus.

The construction of these Hasse invariants has a long history even beyond the classical cases of the modular curve or the Siegel Shimura variety. E. Goren established the existence of partial Hasse invariants for the case of Hilbert-Blumenthal Shimura varieties (\cite{Goren-partial-hasse}). In the split unitary case of signature $(n-1,1)$, Ito constructed Hasse invariants for all Ekedahl-Oort strata (\cite{Ito-hasse}). Finally, Goldring and Nicole constructed a $\mu$-ordinary Hasse invariant for unitary Shimura varieties (\cite{Goldring-Nicole-mu-Hasse}). After the present paper appeared on the ArXiv, some more results were proved. For PEL-cases of type A and C, Boxer constructed Hasse invariants on all Ekedahl-Oort strata. More recently, \cite{Hernandez-hasse} constructs $\mu$-ordinary Hasse invariants in the (unramified) unitary case using crystalline cohomology (a method similar to the one used in \cite{Goldring-Nicole-mu-Hasse}), and \cite{Bij-Hern-hasse} extends Hernandez' result to the ramified case. These approaches all rely on ``semi-linear algebra techniques'' using Dieudonn\'e modules or related structures which become more and more involved the more complicated the moduli problem is.

In this article we introduce a new group theoretic approach to the construction of $\mu$-ordinary Hasse invariants. Since this paper has appeared on ArXiv, our methods have been vastly extended by Goldring and the first author (\cite{Goldring-Koskivirta-Strata-Hasse} and \cite{Goldring-Koskivirta-zip-flags}) to construct Hasse invariants for all Ekedahl-Oort strata and to give several arithmetic applications. See below for a comparative description of these papers.

To be more precise we introduce some notation.

\subsection*{Shimura varieties and $G$-zips}
Let $(\mathbf{G}, X)$ be a Shimura datum of Hodge-type and let $\Scal_K$ be the Kisin-Vasiu integral model of the associated Shimura variety $Sh_K(\mathbf{G},X)$ at a level $K$, hyperspecial at $p$. Denote by $S_K$ the special fiber of $\Scal_K$ and write $G$ for the special fiber of a reductive $\ZZ_{p}$-model of $\mathbf{G}_{\QQ_p}$. We study the $\mu$-ordinary stratum using the Ekedahl-Oort stratification.

Recall that Zhang \cite{ZhangEOHodge} gives a smooth morphism $\zeta\colon S_K \to \GZip^\mu$, where $\GZip^\mu$ is the stack of $G$-zips, defined by Pink-Wedhorn-Ziegler in \cite{Pink-Wedhorn-Ziegler-zip-data} (see also the precursor paper \cite{Moonen-Wedhorn-Discrete-Invariants}) and $\mu$ pertains to the cocharacter attached to the Shimura datum. The geometric fibers of $\zeta$ are called Ekedahl-Oort strata of $S_K$. In this paper we study the open zip stratum $U_\mu\subset \GZip^\mu$ and its inverse image, the generic Ekedahl-Oort stratum $S_{K,\mu}\subset S_K$. It coincides with the $\mu$-ordinary Newton stratum.

Attached to the pair $(G,\mu)$, there is a zip datum $\Zcal:=(G,P,L,Q,M,\varphi)$ (\S\ref{subsec-cochar}), where $L$ is the centralizer of $\mu$ in $G$ and $P$ corresponds to the stabilizer of the Hodge filtration. One attaches to each $\lambda \in X^*(L)$ a line bundle $\Vscr(\lambda)$ on the stack $\GZip^\mu$. Its pull-back $\zeta^*(\Vscr(\lambda))$ coincides with the automorphic line bundle $\Vscr_K(\lambda)$ naturally attached to $\lambda$. For example, there exists $\lambda_\omega \in X^*(L)$ such that $\Vscr_K(\lambda_\omega)$ is the Hodge line bundle $\omega$ on $S_K$.

\subsection*{Hasse invariants}

In this paper, we say that a section $h\in H^0(\GZip^\mu,\Vscr(\lambda))$ is a Hasse invariant if its non-vanishing locus is exactly the $\mu$-ordinary stratum $U_\mu$. As $H^0(\GZip^\mu,\Vscr(\lambda))$ has dimension $\leq 1$ (\Prop~\ref{GlobalSec}~\eqref{item-Glob1}), $h$ is unique up to scalar. There is an explicit integer $N_\mu$ (\Def~\ref{defNeta}) satisfying the following:

\begin{thm1}[\Th~\ref{mainthm}]
If $\lambda \in X^*(L)$ is $\Zcal$-ample, there exists a unique (up to scalar) Hasse invariant $h\in H^0(\GZip^\mu,\Vscr(N_\mu\lambda))$.
\end{thm1}

For the definition of $\Zcal$-ample, see \Def~\ref{Zampledef}. The character $\lambda_{\omega}$ defining the Hodge line bundle is $\Zcal$-ample.

In \cite[\Th~3.2.3]{Goldring-Koskivirta-Strata-Hasse}, a similar result for all strata is proved using a group-theoretical counterpart of a flag space of Ekedahl and Van der Geer. The methods used here to prove \Th~1 differ in many aspects from \loccit; they are based on the study of Ekedahl-Oort strata in the case of a Weil restriction, which we explain below.
%
%
Furthermore, here we do not assume in \Th~1 that $\lambda$ satisfies the condition "orbitally $p$-close" of \loccit Hence \Th~1 gives a stronger result than \loccit for the open zip stratum. Another difference is the fact that we determine explicitly the integer $N_\mu$ in \Th~1, whereas \loccit gives an undetermined integer. This is the smallest integer satisfying the existence of a Hasse invariant. We compute $N_\mu$ for PEL-cases in Subsection~\ref{CalcNmu}.

In particular, we obtain the following corollary (\Cor~\ref{mainShimura}):

\begin{cor1}\label{introthmShimura}
There exists a section $h_K\in H^0(S_K,\omega^{N_\mu})$ whose non-vanishing locus is the $\mu$-ordinary locus of $S_K$.
\end{cor1}

Let $S_K^{\rm min}$ denote the minimal compactification of $S_K$ and continue to denote by $\omega$ the extension of the Hodge bundle on $S_K^{\rm min}$. By a formal argument, the section $h_K$ of Corollary 1 extends uniquely to a section of $\omega^{N_\mu}$ over $S^{\rm min}_K$. Define the $\mu$-ordinary locus $S^{\rm min}_{K,\mu}$ as the non-vanishing locus of this extension. We have the following consequence:

\begin{cor2}
The $\mu$-ordinary locus $S^{\rm min}_{K,\mu}$ is affine.
\end{cor2}

\subsection*{Weil restriction}

Our proof of \Th~1 uses a detailed study of the case when $G=\Res_{\FF_{p^r}/\FF_p}(G_1)$ for some integer $r\geq 1$ and a reductive group $G_1$ over $\FF_{p^r}$. Let $k$ be an algebraic closure of $\FF_p$ and identify $G_k$ with $G_1\times...\times G_r$ for $G_i=\sigma^{i-1}(G_1)$, where $\sigma\in \Gal(k/\FF_p)$ is the geometric Frobenius element. Let $\mu \in X^*(G)$ be a cocharacter and $\Zcal=(G,P,L,Q,M,\varphi)$ its associated zip datum (\S\ref{subsec-cochar}). For each $1\leq j \leq r$, we define a zip datum $\Zcal_j=(G_j,P'_j,L'_j,Q'_j,M'_j,\varphi^r)$ (\S\ref{subsec-Zj}). Our main technical result, which is of independent interest for the study of zip strata, is the following kind of ``Shapiro's lemma''. 

\begin{thm2}[\Th~\ref{bijorbits}]
The map $C\mapsto C\cap G_{j}$ defines a
bijection between zip strata for $\Zcal$ in $G$ intersecting $G_j$ and zip strata for $\Zcal_j$ in $G_j$. Furthermore one has $\codim_G(C) = \codim_{G_j}(C\cap G_j)$ for all such $C$.
\end{thm2}

Moreover, there is a relation between Hasse invariants for $G$ and Hasse invariants for each factor (\Prop~\ref{propHasseWeil}).

We now give an overview of the paper. In \S1, we recall basic facts about the stack $\GZip^\Zcal$. Then, we prove some general results on equivariant Picard groups in \S2. The next section is dedicated to results about the open zip stratum. In particular we determine the stabilizer of $1$, used to define the integer $N_\mu$. In \S4, we study the case of a Weil restriction and prove \Th~2 above. Finally, we prove \Th~1 and \Cor~1 in the last section, combining the results of the previous sections. 

\section*{Acknowledgments}
The first author is grateful to Wushi Goldring for valuable remarks on the paper. He also wishes to thank David Helm for fruitful discussions.

\section{$G$-zips}
\subsection{Stack of $G$-zips} \label{subsec-GZip}
Let $p$ be a prime number, and $q$ a power of $p$. We fix an algebraic closure $k$ of $\FF_q$. In this paper we define a zip datum over $\FF_q$ to be a tuple $\Zcal=(G,P,L,Q,M,\varphi)$ where $G$ is a connected reductive group over $\FF_q$, $\varphi:G\to G$ is the relative $q$-power Frobenius homomorphism, $P,Q\subset G$ are parabolic subgroups of $G_k$, $L\subset P$ and $M\subset Q$ are Levi subgroups of $P$ and $Q$ respectively. We suppose that $\sigma(L) = M$, where $\sigma(\ )$ denotes pullback by $q$-power absolute Frobenius morphism. The zip group $E$ is defined by
\begin{equation}\label{zipgroup}
E:=\{(x,y)\in P\times Q, \varphi(\overline{x})=\overline{y}\}
\end{equation}
where $\overline{x}\in L$ and $\overline{y}\in M$ are the Levi components  of $x$ and $y$. We let $G\times G$ act on $G$ via $(a,b)\cdot g:=agb^{-1}$ and we let $E$ act on $G$ by restricting this action. The stack of $G$-zips is the quotient stack (\cite{PinkWedhornZiegler-F-Zips-additional-structure}):
\begin{equation}
\GZip^\Zcal\simeq \left[E \backslash G \right].
\end{equation}

\subsection{Frame}\label{subsec frame}
A frame for $\Zcal$ is a triple $(B,T,z)$ where $(B,T)$ is a Borel pair in $G_k$ and $z\in G(k)$ satisfying:
\begin{enumerate}[(i)]
\item $B\subset Q$, ${}^z B \subset P$,
\item $\varphi({}^z B\cap L)=B\cap M$, $\varphi({}^z T)= T$.
\end{enumerate}
In particular, these conditions imply ${}^z T\subset L$. Frames always exist by \cite[3.7]{Pink-Wedhorn-Ziegler-zip-data}. For a frame $(B,T,z)$, we use the following notation.
\begin{enumerate}[(1)]
\item $\Phi\subset X^*(T)$ denotes the set of $T$-roots of $G$.
\item $\Phi_+ \subset \Phi$ is the set of positive roots with respect to $B$, i.e such that $U_\alpha \subset B$.
\item $\Delta\subset \Phi_+$ is the set of positive simple roots.
\item Denote by $W:=W(G_k,T)$ the Weyl group of $G_k$. For $\alpha \in \Phi$, let $s_\alpha \in W$ be the corresponding reflection. Then $(W,\{s_\alpha\}_{\alpha \in \Delta})$ is a Coxeter group and we denote by $\ell :W\to \NN$ the length function.
\item For a subset $K\subset \Delta$, Let $W_K \subset W$ be the subgroup generated by $\{s_\alpha, \ \alpha\in K\}$. Let $w_0\in W$ be the longest element in $W$ and $w_{0,K}$ the longest element in $W_K$.
\item If $R\subset G$ is a parabolic subgroup containing $B$ and $D$ is the unique Levi subgroup of $R$ containing $T$, then the type of $R$ (or of $D$) is the unique subset $K\subset \Delta$ such that $W(D,T)=W_K$. The type of an arbitrary parabolic $R$ is the type of its unique conjugate containing $B$. Let $I\subset \Delta$ (resp. $J\subset \Delta$) be the type of $P$ (resp. $Q$).
\item For $K\subset \Delta$, let ${}^K W$ (resp. $W^K$) be the set of elements $w\in W$ which are of minimal length in the coset $W_K w$ (resp. $wW_K$).
\end{enumerate}
 We say that $(B,T,z)$ is an $\FF_q$-frame if $B,T$ are defined over $\FF_q$. In this case one has $z\in N_G(T)$, so $z$ gives rise to an element $z\in W$.
\subsection{$E$-orbits} \label{subsec stratif}
For $w\in W$ we choose a representative $\dot{w}\in N_G(T)$, such that $(w_1w_2)^\cdot = \dot{w}_1\dot{w}_2$ whenever $\ell(w_1 w_2)=\ell(w_1)+\ell(w_2)$ (this is possible by choosing a Chevalley system, see \cite[ XXIII, \S6]{SGA3}). We often write simply $w$ instead of $\dot{w}$. For $h\in G(k)$, let $\Ocal_\Zcal(h)$ be the $E$-orbit of $h$ in $G$. By \cite[\Th~7.5]{Pink-Wedhorn-Ziegler-zip-data}, there is a bijection:
\begin{equation}\label{param}
{}^I W \to \{E \textrm{-orbits in }G\}, \quad w\mapsto G_w:=\Ocal(z\dot{w}).
\end{equation}
Furthermore, for all $w\in {}^I W $, one has $\dim(G_w)= \ell(w)+\dim(P)$. Endow the locally closed subset $G_w$ with the reduced structure, and define the corresponding zip stratum of $\GZip^\Zcal$ by $Z_w:= \left[E\backslash G_w \right]$. We will denote by $U_\Zcal \subset \GZip^\Zcal$ the unique open zip stratum corresponding to $w_{0,I}w_0$, the longest element in ${}^IW$. When $\Zcal=\Zcal_\mu$ for some $\mu \in X^*(G)$, we write $U_\mu:=U_{\Zcal_\mu}$.

\subsection{Cocharacters} \label{subsec-cochar}
A cocharacter $\mu:\GG_{m,k}\to G_k$ defines a pair of opposite parabolics $P_\pm(\mu)$ such that $P_+(\mu)\cap P_-(\mu)=L(\mu)$ is the centralizer of $\mu$. The parabolic $P_+(\mu)$ consists of elements $g\in G$ such that the limit
\begin{equation}\label{limparab}
\lim_{t\to 0}\mu(t)g\mu(t)^{-1}
\end{equation}
exists, i.e such that the map $\GG_{m,k} \to G_{k}$,  $t\mapsto\mu(t)g\mu(t)^{-1}$ extends to a morphism of varieties $\AA_{k}^1\to G_{k}$. The unipotent radical of $P_+(\mu)$ is the set of such elements $g$ for which this limit is $1\in G(k)$.

Define $P:=P_-(\mu)$, $Q:=\sigma\left( P_+(\mu) \right)$, $L:=L(\mu)$ and $M:= \sigma(L(\mu))$. The tuple $\Zcal_\mu:=(G,P,L,Q,M,\varphi)$ is a zip datum attached to the cocharacter $\mu$. In the following we consider only $G$-zips arising in this way. We write simply $\GZip^\mu$ for $\GZip^{\Zcal_\mu}$. Replacing $\mu$ by a conjugate cocharacter does not change the isomorphism class of $\GZip^{\mu}$. Hence the following remark shows that it is harmless to assume that $\Zcal$ admits an $\FF_q$-frame.

\begin{rmk}\label{framecochar}
Let $\mu:\GG_{m,k}\to G_k$ be a cocharacter and $\Zcal_\mu:=(G,P,L,Q,M,\varphi)$ the zip datum attached to $\mu$.
\begin{enumerate}
\item \label{itemframe1} There exists a conjugate cocharacter $\mu':=\ad(g)\circ \mu$ (for $g\in G$) such that $\Zcal_{\mu'}$ admits an $\FF_q$-frame.
\item \label{itemframe2} Assume $(B,T)$ is a Borel pair defined over $\FF_q$ such that $B\subset Q$. Then $(B,T,w_0 w_{0,I})$ is an $\FF_q$-frame.
\end{enumerate}
\end{rmk}

Let $G_1,G_2$ be connected $\FF_q$-reductive groups and $\mu_1\in X_*(G_1)$, $\mu_2\in X_*(G_2)$ cocharacters. Let $f:G_1\to G_2$ be an injective homomorphism defined over $\FF_p$, such that $\mu_2=f_k\circ \mu_1$. Then $f$ induces naturally a morphism of stacks
\begin{equation}
f^{\#}:\GoneZip^{\mu_1}\to \GtwoZip^{\mu_2}.
\end{equation}
For $i=1,2$, denote by $P_i,Q_i$ the parabolic subgroups of $G_i$ attached to $\mu_i$. Using \eqref{limparab}, it is clear that $f(\square_1)\subset \square_2$ for $\square=P,L,Q,M$ and also $f(R_u(P_1))\subset R_u(P_2)$, $f(R_u(Q_1))\subset R_u(Q_2)$. It follows easily that $P_1=f^{-1}(P_2)$, $Q_1=f^{-1}(Q_2)$.

\section{Picard groups}
\subsection{Equivariant Picard group}

In this section, $G$ denotes an arbitrary connected linear algebraic
group over an algebraically closed field $k$. A variety is an integral $k$-scheme of finite type. A $G$-scheme is a $k$-scheme endowed with a $G$-action $a: G \times X \to X$.
Let $\mathscr{L}$ be a line bundle on a $G$-scheme $X$. Define the projections $p_{23}\colon G\times G\times X\to G\times X$, $(g,h,x)\mapsto (h,x)$ and $p_2:G\times X\to X$, $(g,x)\mapsto x$. Finally, write $\mu_G$ for the multiplication map $G\times G\to G$.

\begin{definition}
A $G$-linearization of $\mathscr{L}$ is an isomorphism $\phi : a^*(\mathscr{L})\to p_{2}^*(\mathscr{L})$ satisfying the cocycle condition
$$ p_{23}^*(\phi) \circ (id_G \times a)^*(\phi) =(\mu_G \times id_X)^*(\phi).$$
\end{definition}

We denote by $\Pic^{G}(X)$ the group of isomorphism classes of $G$-linearized line bundles on $X$. Forgetting the $G$-linearization induces a natural map $\Pic^{G}(X) \to \Pic(X)$, whose image is the subgroup $\Pic_{G}(X)\subset \Pic(X)$ of $G$-linearizable line bundles. The group $\Pic^{G}(X)$ can be identified with the Picard group of the quotient stack $\left[G \backslash X\right]$. Then $\Pic^G(X) \to \Pic(X)$ is the homomorphism given by pull back by the projection $X \to [G \backslash X]$.

\subsection{The general result}
For any $k$-scheme $X$, define $\Ecal(X) := \mathcal{O}(X)^{\times} / k^{\times}$. If $X$ is an integral $k$-scheme of finite type, $\Ecal(X)$ is a finitely generated free abelian group (\cite[\S 1.3]{Knop-Kraft-Vust-G-variety}). If $G$ is an algebraic group, the natural map $X^*(G)\to \Ecal(G)$ is an isomorphism, by \loccit

Denote by $H^1_{\rm alg}(G,\Ocal(X)^\times)\subset H^1(G,\Ocal(X)^\times)$ the subgroup of classes of algebraic cocycles, i.e. cocycle maps $G(k) \to \Ocal(X)^\times$ induced by an algebraic morphism $G \times X \to \GG_m$.

\begin{theorem}\label{exseqgen}
Let $G$ be a smooth algebraic group, and $X$ an irreducible $G$-variety. There are exact sequences:
\[
0\to \frac{\GG_m(X)^G}{k^\times}\to \Ecal(X)^{G}\to X^{*}(G)\to H^1_{\rm alg}(G,\GG_m(X)) \to H^{1}(\pi_0(G)(k),\Ecal(X))
\]
\[
0\to H^1_{\rm alg}(G,\GG_m(X)) \to \Pic^{G}(X) \to \Pic(X)
\]
If $X$ is normal and $G$ connected, the second exact sequence has an extension by a map $\Pic(X) \to \Pic(G)$.
\end{theorem}

\begin{proof}
This is \cite[\Prop~2.3, \Lem~2.2]{Knop-Kraft-Vust-G-variety} in characteristic $0$. The same proof applies to arbitrary characteristic.
\end{proof}

\subsection{Some consequences}

\begin{corollary}\label{exseq}
Let $G$ be a smooth connected algebraic group $X$ a normal irreducible $G$-variety. There is an exact sequence of abelian groups
\[
1 \to k^{\times} \to (\mathcal{O}(X)^{\times})^G \to \Ecal(X) \to X^{*}(G) \to \Pic^{G}(X) \to \Pic(X) \to \Pic(G)
\]
\end{corollary}

\begin{proof}
As $G$ is connected, it acts trivially on the discrete group $\Ecal(X)$.
\end{proof}

\begin{corollary}\label{exactseqs}
Let $G$ be a smooth connected algebraic group and $H$ a smooth subgroup of $G$. There is an exact sequence of abelian groups
$$0 \to \Ecal(G/H) \to X^*(G) \to X^{*}(H) \to \Pic(G/H) \to \Pic(G).$$
\end{corollary}

\begin{proof}
Apply \Th~\ref{exseqgen} for $X = G$ and $G = H$. Note that $H$ acts trivially on $\Ecal(G)=X^{*}(G)$. It follows that $H^{1}(\pi_0(H),\Ecal(G)) = \Hom(\pi_0(H),\Ecal(G)) = 0$ because $\pi_0(H)$ is finite and $\Ecal(G)$ is torsion-free.
\end{proof}

\begin{corollary}\label{surject}
Let $G$ be a connected linear algebraic $\mathbb{F}_{q}$-group. There is an exact sequence
$$0 \to X^{*}(G)\to X^{*}(G)\to \Hom(G(\mathbb{F}_{q}),k^{\times}) \to \Pic(G)$$
where the first map is $\chi \mapsto \sigma \cdot \chi-\chi$.
\end{corollary}

\begin{proof}
Apply \Cor~\ref{exactseqs} for $H:=G(\FF_q)$ and note that the Lang-Steinberg map $G\to G$, $x\mapsto \varphi(x)x^{-1}$ induces an isomorphism of varieties $G/H \simeq G$.
\end{proof}



\section{The open orbit}
\subsection{Intersection of parabolic subgroups}
Let $k$ be an algebraically closed field and $G$ a reductive group over $k$. Recall the following result:
\begin{proposition}\label{InterParab}
Let $P$ and $Q$ be two parabolic subgroups in $G$ with unipotent radicals $U$ and $V$, respectively. Let $T\subset P\cap Q$ be a maximal torus and let $L\subset P$ and $M\subset Q$ denote the Levi subgroups containing $T$.
\begin{assertionlist}
\item\label{InterParab1}
The subgroups $P\cap Q$, $L\cap M$, $L\cap V$, $M\cap U$, $U\cap V$ are smooth and connected.
\item\label{InterParab2}
The group $\left(P\cap Q\right).U$ is a parabolic subgroup of $G$
contained in $P$, with Levi subgroups $L\cap M$.
\item\label{InterParab3}
Any element $x\in P\cap Q$ can be written uniquely as a product $x=abcd$,
with $a\in L\cap M$, $b\in L\cap V$, $c\in M\cap U$, $d\in U\cap V$.
\end{assertionlist}
\end{proposition}

\begin{proof}
 The smoothness of $P\cap Q$ follows from \cite[XXVI, \Lem~4.1.1]{SGA3}. This implies the smoothness of the other subgroups. For the rest, see \cite[\Prop~2.1]{digne-michel}.
\end{proof}

The last statement means that $P\cap Q$ is the product of the varieties $L\cap M$, $L\cap V$, $M\cap U$, $U\cap V$. In particular:

\begin{cor}\label{contained}
Retain the notation of \Prop~\ref{InterParab}.
If $U\cap V=M\cap U=\{1\}$, then $P\cap Q\subset L$.
\end{cor}

For a parabolic subgroup $P$ and a Levi subgroup $L\subset P$, denote by $\theta_{L}^{P}:P\to L$ the natural projection modulo the unipotent radical of $P$.

\begin{cor}\label{CorReduct}
Retain the notation of \Prop~\ref{InterParab}.
\begin{assertionlist}
\item \label{item-cor1} For all $x\in P\cap Q$, one has $\theta_{L}^{P}(x)\in L\cap Q$ and $\theta_{M}^{Q}(x)\in P\cap M$.
\item \label{item-cor2} For all $x\in P\cap Q$, one has $\theta_{L}^{P}(\theta_{M}^{Q}(x))=\theta_{M}^{Q}(\theta_{L}^{P}(x))\in L\cap M$.
\item \label{item-cor4} Assume $T\subset B\subset P\cap Q$ for some Borel $B$. Then $P\cap Q$ is a parabolic with Levi $L\cap M$ and for all $x\in P\cap Q$, one has $\theta_{L}^{P}(\theta_{M}^{Q}(x))=\theta_{M}^{Q}(\theta_{L}^{P}(x))=\theta_{L\cap M}^{P\cap Q}(x)$.
\item \label{item-cor3} Assume $G$ is defined over $\FF_q$ and let $\varphi:G\to G$ the $q$-th power Frobenius. Then $\varphi \left( \theta_{L}^{P}(x) \right)=  \theta_{\sigma L}^{\sigma P}(\varphi(x))$.
\end{assertionlist}
\end{cor}

\begin{proof}
Using the notation of \Prop~\ref{InterParab}\eqref{InterParab3}, write $x=abcd \in P \cap Q$. Then $\theta_{L}^{P}(x)=ab\in L\cap Q$ and similarly $\theta_{M}^{Q}(x)=ac\in P\cap M$ which shows \eqref{item-cor1}. This implies $\theta_{L}^{P}(\theta_{M}^{Q}(x))=a=\theta_{M}^{Q}(\theta_{L}^{P}(x))$ and hence \eqref{item-cor2}. The first part of \eqref{item-cor4} is \Prop~\ref{InterParab}\eqref{InterParab2}. For the second part, write $x=\theta_{L}^{P}(x)u$ with $u\in R_u(P)\subseteq R_u(P\cap Q)$. Now $\theta_{L}^{P}(x)=\theta_{M}^{Q} \left( \theta_{L}^{P}(x) \right) v$ with $v\in R_u(Q)\subseteq R_u(P\cap Q)$. The result follows. The last assertion is clear. 
\end{proof}

\subsection{Stabilizer}\label{stabilizer}
We fix a cocharacter $\mu:\GG_{m,k}\to G_k$ and we assume that $\Zcal_\mu$ admits an $\FF_q$-frame $(B,T,z)$. By \Rmk~\ref{framecochar}\eqref{itemframe2}, we can take $z:=w_0 w_{0,I}$. Let $B_-$ be the unique Borel subgroup such that $B\cap B_-=T$. Note that $B_-\subset P$. The maximal element of ${}^I{W}$ is $\eta = w_{0,I}w_0=z^{-1}$. By \eqref{param}, the unique open $E$-orbit in $G$ is $G_{\eta}=E\cdot (z\eta)=E\cdot 1$.

Since $\dim(E)=\dim(G)$, the stabilizer $\Stab_E(1)\subset E$ is a finite group scheme (usually not smooth), that we now determine. We denote by $S_{\rm red}$ its underlying reduced subgroup scheme. By definition, one has
\begin{equation}
\Stab_E(1) = \{(x,y) \in E, \ x = y\}.
\end{equation}
The first projection $E\to P$ identifies $\Stab_E(1)$ with a subgroup $S\subset P\cap Q$. Define:
\[
P_{0}:=\bigcap_{i\in\mathbb{Z}}\sigma^i P\quad,\quad Q_{0}:=\bigcap_{i\in\mathbb{Z}}\sigma^i Q\quad,\quad L_{0}:=\bigcap_{i\in\mathbb{Z}}\sigma^i L.
\]
Since $B\subset Q_0$ and $B_-\subset P_{0}$ and $(B,T)$ is defined over $\FF_q$, we see that $P_0$ and $Q_0$ are opposite parabolic $\mathbb{F}_{q}$-subgroups such that $P_0\cap Q_0=L_{0}$. The type of $P_0$ (resp. $Q_0$) is $\bigcap_i \sigma^i I$ (resp. $\bigcap_i \sigma^i J$).

\begin{lemma}\label{ContainQ0} \ 
\begin{enumerate}
\item \label{item-contain1} One has $Q_{0}\cap P\subset L$.
\item \label{item-contain2} One has $S\subset Q_0 \cap L$.
\item \label{item-contain3} One has $S_{\rm red}=S\cap L_0=L_{0}(\mathbb{F}_{q})$ (considered as a finite constant group scheme).
\end{enumerate}
\end{lemma}

\begin{proof}
Assertion \eqref{item-contain1} follows from \Cor~\ref{contained}, because $L_{0}\cap R_{u}(P) \subseteq L \cap R_{u}(P) =\{1\}$ and $R_{u}(Q_{0})\cap R_{u}(P) \subseteq  Q \cap R_{u}(P) =\{1\}$.

We now prove \eqref{item-contain2}. Let $x\in S$ be an arbitrary element. By definition of $E$, one has $\varphi(\theta_{L}^{P}(x))=\theta_{M}^{Q}(x)$. Since $x\in P\cap Q$, one has $\theta_{L}^{P}(x)\in Q$ by \Cor~\ref{CorReduct}. Hence $\theta_{M}^{Q}(x)\in\sigma(Q)$ and we deduce $x\in\sigma(Q)$
because $R_{u}(Q)\subset R_{u}(B)\subset\sigma(Q)$. Now we can apply
the same argument to $\sigma(Q)$ to show $x\in\sigma^{2}(Q)$. Continuing
this process, we get $x\in Q_{0}$. Since $S\subset P$, we also have
$S\subset L$ by \eqref{item-contain1}. This shows \eqref{item-contain2}.

To prove \eqref{item-contain3}, we show first $S_{\rm red}\subset L_{0}$. It suffices to show $S(k)\subset L_{0}(k)$. Let $x\in S(k)$. Again $\varphi(\theta_{L}^{P}(x))=\theta_{M}^{Q}(x)$, so $\theta_{M}^{Q}(x)\in P$ by \Cor~\ref{CorReduct}. We deduce that $\theta_{L}^{P}(x)\in\sigma^{-1}(P)$ and then $x\in\sigma^{-1}(P)$.
Repeating the argument, we get $x\in P_{0}$. Since $S\subset Q_0$ by \eqref{item-contain2}, we conclude $S_{\rm red}\subset L_{0}$. For $x\in L_{0}$, $x\in S_{\rm red}$ if and only if $\varphi(x)=x$, so $S_{\rm red}=L_{0}(\mathbb{F}_{q})$. Since the Lang-Steinberg map $L_{0}\to L_{0}$, $x\mapsto x^{-1}\varphi(x)$ is \'{e}tale, it follows that the algebraic group $S\cap L_{0}=\left\{ x\in L_{0},\ \varphi(x)=x\right\}$ is smooth, equal to the constant group $L_{0}(\mathbb{F}_{q})$, hence $S\cap L_{0}=S_{\rm red}=L_0(\FF_q)$.
\end{proof}

Denote by $S^\circ \subset S$ is the identity component of $S$. Since $S$ is a finite group-scheme over a perfect field, $S\simeq S^\circ \rtimes S_{\rm red}$ and in particular 
\begin{equation}\label{charex}
X^*(S) \simeq X^*(S_{\rm red}) \times X^*(S^\circ).
\end{equation}

\begin{lemma}\label{ProjectS} \ 
\begin{enumerate}
\item \label{item-ProjS1} If $x\in S$, then $\theta_{L_{0}}^{Q_{0}}(x)\in S_{\rm red}$.
\item \label{item-ProjS2} One has $S^\circ \subset R_u(Q_0)\cap L$.
\item \label{item-ProjS3} One has $X^*(S^\circ)=0$, in particular $X^*(S)= X^*(S_{\rm red})=\Hom(L_0(\FF_q),k^\times)$.
\end{enumerate}
\end{lemma}

\begin{proof}
We prove first \eqref{item-ProjS1}. Let $x\in S$ be an element. One has $\varphi(\theta_{L}^{P}(x))=\theta_{M}^{Q}(x)$ by definition and $x\in Q_{0}\cap L$ by \Lem~\ref{ContainQ0}\eqref{item-contain2}. Hence $\theta_{L}^{P}(x)=x$ and $\varphi(x)=\theta_{M}^{Q}(x)$. Using \Cor~\ref{CorReduct}, we deduce
\[
\varphi(\theta_{L_{0}}^{Q_{0}}(x))=\theta_{L_{0}}^{Q_{0}}(\varphi(x))=\theta_{L_{0}}^{Q_{0}}(\theta_{M}^{Q}(x)) = \theta_{M}^{Q}(\theta_{L_{0}}^{Q_{0}}(x)) = \theta_{L_{0}}^{Q_{0}}(x) \]
It follows that $\theta_{L_{0}}^{Q_{0}}(x)\in L_{0}(\mathbb{F}_{q})=S_{\rm red}$, which proves \eqref{item-ProjS1}. We obtain a group homomorphism $\theta_{L_{0}}^{Q_{0}}:S\to S_{\rm red}$, which is necessarily trivial on $S^\circ$, hence $S^\circ \subset \Ker(\theta_{L_{0}}^{Q_{0}})=R_u(Q_0)$, which proves \eqref{item-ProjS2}. Finally, \eqref{item-ProjS3} is a direct consequence of \eqref{item-ProjS2}.
\end{proof}

Define a group homomorphism $\label{EqDefineZeta}
\zeta\colon X^{*}(L_{0})\to X^{*}(L_{0}), \ \chi \mapsto \chi - \chi\circ\varphi$. It is clear that $\zeta$ is injective.

\begin{cor}\label{zeta}
Assume $\Pic(L_0)=0$ (e.g., if the derived group of $G$ is simply connected). Then one has
\begin{enumerate}
\item \label{item-zeta1} $\Coker(\zeta)=X^{*}(S)$.
\item \label{item-zeta2} The order of $X^{*}(S)$ is $|\det(\zeta)|$.
\end{enumerate}
\end{cor}

\begin{proof}
Assertion \eqref{item-zeta1} follows from \Cor~\ref{surject} and \eqref{item-zeta2} is an immediate consequence.
\end{proof}

\begin{definition}\label{defNeta}
We define the integer $N_\mu$ as the exponent of the finite group $X^*(S) = \Hom(L_0(\FF_q),k^\times)$.
\end{definition}

\begin{cor}\label{HasseSplit}
Assume $G$ is split over $\mathbb{F}_{q}$ and that the derived group of $G$ is simply connected. Then $N_\mu=q-1$.
\end{cor}

\begin{proof}
Since $G$ is $\FF_q$-split, one has $\zeta = 1-q$, so the result follows.
\end{proof}

\subsection{Line bundles on $\GZip^\mu$}\label{subsecline}
The first projection $E\to P$ induces an identification $X^*(E)=X^*(P)=X^*(L)$. For a character $\lambda \in X^*(L)$, let $\Vscr(\lambda)$ be the line bundle on $\GZip^\Zcal \simeq [E\backslash G]$ attached to $\lambda$ via $X^*(E)\to \Pic^E(G)$. A global section $s\in H^0(\GZip^\Zcal,\Vscr(\lambda))$ is a regular map $s: G_k\to \AA_k^1$ satisfying the condition
\begin{equation}\label{eqsection}
s(\varepsilon \cdot g)=\lambda(\varepsilon)s(g), \qquad \forall g\in G, \ \varepsilon\in E.
\end{equation}
Recall that $U_\Zcal$ denotes the unique open zip stratum in $\GZip^\Zcal$. Similarly, a section $s\in H^0(U_\mu,\Vscr(\lambda))$ is a regular map $s: G_\eta\to \AA_k^1$ satisfying \eqref{eqsection}.

\begin{proposition}\label{GlobalSec}
Let $\lambda \in X^*(L)$ be a character. 
\begin{assertionlist}
\item \label{item-Glob1} One has $\dim_k(H^{0}(\GZip^\mu,\Vscr(\lambda)))\leq \dim_k(H^{0}(U_\mu,\Vscr(\lambda))) \leq 1$.
\item \label{item-Glob2} The space $H^{0}(U_\mu,\Vscr(N_\mu \lambda))$ has dimension one.
\end{assertionlist}
\end{proposition}

\begin{proof}
The inclusion $U_\mu \subset \GZip^\mu$ induces an injection $H^{0}(\GZip^\mu,\Vscr(\lambda)) \to H^{0}(U_\mu,\Vscr(\lambda))$. By \eqref{eqsection}, an element $f\in H^{0}(U_\mu,\Vscr(\lambda))$ is uniquely determined by the value of $f$ at a given point of $G_\eta$, which proves \eqref{item-Glob1}.

The space $H^{0}(U_\mu,\Vscr(\lambda))$ is nonzero if and only if $\lambda :L\to k^\times$ is trivial on the subgroup $S\subset L$. Since $N_\mu \lambda$ induces the trivial character on $S$, we deduce \eqref{item-Glob2}.
\end{proof}

\section{Weil restriction and $G$-zips}

\subsection{Notations}\label{Weilres}

Let $r\geq 1$ be an integer. Let $G_1$ be a connected reductive group over $\mathbb{F}_{q^r}$ and $G=\textrm{Res}_{\mathbb{F}_{q^r}/\mathbb{F}_q}(G_1)$. Let $\sigma\in\mbox{Gal}(k/\mathbb{F}_{q})$ be the $q$-th power arithmetic Frobenius. Over $k$, the group $G$ decomposes as a product
\begin{equation}
G_k=G_{1}\times\cdots\times G_{r}\label{prodG}
\end{equation}
where $G_{i}=\sigma^{i-1}(G_{1})$. The Frobenius sends $G_{i}$ onto $G_{i+1}$ (indices taken modulo $r$).

Let $\mu:\GG_{m,k}\to G_k$ be a cocharacter and $\Zcal_\mu:=(G,P,L,Q,M,\varphi)$ its attached zip datum. Assume further that $(B,T)$ is an $\FF_q$-Borel pair such that $B\subset Q$. For $z:=w_0 w_{0,I}$, the triple $(B,T,z)$ is an $\FF_q$-frame for $\Zcal_\mu$ (\Rmk~\ref{framecochar}~\eqref{itemframe2}). For each $\square=P,L,Q,M,B,T$, one has a decomposition $\square=\prod_{i=1}^r \square_i$. Since $(B,T)$ is defined over $\FF_q$, one has $\sigma(B_i) = B_{i+1}$ and $\sigma(T_i) = T_{i+1}$. By definition one has $\sigma(L_i)=M_{i+1}$.

The Weyl group $W:=W(G,T)$ decomposes into a product (as a Coxeter group)
\[
W = W_{1} \times \cdots \times W_{r},
\]
where $W_i:=W(G_i,T_i)$. Let $w_{0,i}$ be the element of maximal length in $W_i$. The Frobenius induces an automorphism of Coxeter groups of $W$ again denoted by $\sigma$ and we have $\sigma(W_i) = W_{i+1}$.

Let $I$ be the type of $P$. Then ${}^I W$ decomposes as ${}^I W={}^{I_1} W_1\times\cdots\times {}^{I_r} W_r$ where $I_{i}\subset \Delta_{i}$ is the type of the $P_{i}$. Then one has $z=(z_1,...,z_r)$ where $z_j=w_{0,j} w_{0,I_j}$ for all $j=1,...,r$. Using the parametrization \eqref{param}, the $E$-orbits of codimension one in $G$ are given by
\begin{equation}
\label{Eorbs}
C_{j,\alpha}:=E\cdot\left(1,\dots,1,\dot{w}_{0,j} \dot{s}_{\alpha} \dot{w}_{0,j},1,\dots,1\right), \qquad 1\leq j \leq r, \ \alpha \in \Delta_j \setminus I_j.
\end{equation}

\begin{rmk}\label{DescribERes}
An element of $E$
can be written in the form $\left(\left(x_{1},\dots,x_{r}\right),\left(y_1,y_2,\dots,y_{r}\right)\right)$ with $x_i \in P_i$, $y_i \in Q_{i}$ and $\varphi(\theta_{L_i}^{P_i}(x_{i}))=\theta_{M_{i+1}}^{Q_{i+1}}(y_{i+1})$ for all $1 \leq i \leq r$ (indices taken modulo $r$).
\end{rmk}

\subsection{The zip datum $\Zcal_j$} \label{subsec-Zj}

For $j=1,...,r$, define parabolic subgroups in $G_{j}$ by
\[
P'_{j}=\bigcap_{i=0}^{r-1}\sigma^{-i}(P_{i+j}) \quad\mbox{and}\quad Q'_{j}= \bigcap_{i=0}^{r-1}\sigma^{i}(Q_{j-i})
\]
where the index $i$ of $P_i,Q_i$ is taken modulo $r$. Clearly $B^-_j\subset P'_j$ and $B_j\subset Q'_j$ since $B$ is defined over $\mathbb{F}_{q}$. The Levi subgroups of $P'_j$ and $Q'_j$ containing $T_j$ are respectively
\[
L'_{j}=\bigcap_{i=0}^{r-1}\sigma^{-i}(L_{i+j}) \quad\mbox{and}\quad M'_{j}= \bigcap_{i=0}^{r-1}\sigma^{i}(M_{j-i})
\]
\begin{lemma}
The tuple $\Zcal_j:=(G_j,P'_j,L'_j,Q'_j,M'_j,\varphi^r)$ is a zip datum over $\FF_{q^r}$.
\end{lemma}
\begin{proof}
This follows from the relations 
\begin{equation}
\sigma^r L'_{j}=\bigcap_{i=0}^{r-1}\sigma^{r-i}(L_{i+j})=\bigcap_{i=0}^{r-1}\sigma^{i+1}(L_{j-(i+1)})=\bigcap_{i=0}^{r-1}\sigma^{i}(L_{j-i})=M'_{j}.
\end{equation}
\end{proof}
Denote by $E'_j\subset P'_j\times Q'_j$ the attached zip group.
For $(x,y)\in E'_j$ write $\overline{x}:= \theta_{L'_j}^{P'_j}(x)$ and set
\begin{equation}\label{injEprime}
\begin{aligned}
u_j(x,y) &:= (\varphi^{r-j+1}(\overline{x}),...,\varphi^{r-1}(\overline{x}), x,\varphi(\overline{x}),\dots,\varphi^{r-j}(\overline{x}))\in G \\
v_j(x,y) &:=(\varphi^{r-j+1}(\overline{x}),...,\varphi^{r-1}(\overline{x}), y,\varphi(\overline{x}),\dots,\varphi^{r-j}(\overline{x}))\in G\\
\gamma_j(x,y)&:=(u_j(x,y),v_j(x,y)) \in E.
\end{aligned}
\end{equation}
This gives an injective group homomorphism $\gamma_j:E'_j\to E$. Consider $G_j$ as a subgroup of $G$ by identifying it with $\{1\} \times \dots \times \{1\} \times G_j \times \{1\} \times \dots \times \{1\}$. Define
\begin{equation}
E^{(j)}:=\{\varepsilon \in E, \ \varepsilon \cdot G_j =G_j\}.
\end{equation}
It is clear that $\gamma_j(E'_j)\subset E^{(j)}$.

\begin{rmk}\label{StabG1}
Let $\varepsilon=\left(\left(x_{1},\dots,x_{r}\right),\left(y_1,\dots,y_{r}\right)\right)$ be an element of $E$. The following are equivalent.
\begin{equivlist}
\item $\varepsilon \in E^{(j)}$
\item There exists $g_j \in G_j$ such that $\varepsilon\cdot g_j \in G_j$.
\item $x_i = y_i$ for all $i\neq j$.
\end{equivlist}
\end{rmk}

For each $1\leq d \leq r$, we define a parabolic subgroups $\Ptilde_{d,j}$ in $G_{d+j-1}$ and $\Qtilde_{d,j}$ in $G_{d+j}$  by intersecting Galois translates of the parabolic subgroups $P_i$, $Q_i$ as follows:
\begin{equation}
\Ptilde_{d,j} := \bigcap_{i=d-1}^{r-1}\sigma^{d-i-1}(P_{i+j})\qquad \textrm{and}\qquad \Qtilde_{d,j} := \bigcap_{i=0}^{d-1}\sigma^{i}(Q_{d+j-i})
\end{equation}
where the index $m$ in $P_m$ and $Q_m$ is taken modulo $r$. Note that by definition one has $\Ptilde_{1,j} = P'_{j}$ and $\Qtilde_{r,j} = Q'_j$. One has the formulas:
\begin{align}
&\Ptilde_{d,j}=\sigma^{-1}(\Ptilde_{d+1,j})\cap P_{d+j-1} \quad \textrm{ for all }1\leq d \leq r-1, \textrm{ for all }j \\
&\Qtilde_{d,j}=\sigma(\Qtilde_{d-1,j})\cap Q_{d+j} \quad \quad \ \ \textrm{ for all }2\leq d \leq r, \textrm{ for all }j.
\end{align}

\begin{lemma}\label{technical} \ 
\begin{assertionlist}
\item \label{item-tech1} Let $1\leq j \leq r$ and $\varepsilon=\left(\left(x_{1},\dots,x_{r}\right),\left(y_1,\dots,y_{r}\right)\right)\in E^{(j)}(k)$. Then for all $1 \leq d\leq r$, one has $x_{d+j-1}\in \Ptilde_{d,j}$ and $y_{d+j}\in \Qtilde_{d,j}$, where the index $m$ in $x_m$, $y_m$ is taken modulo $r$. In particular, $x_j\in \ P'_{j}$ and $y_j \in Q'_j$.
\item \label{item-tech2} Furthermore, one has $(x_j,y_j)\in E'_j(k)$.
\end{assertionlist}
\end{lemma}

\begin{proof}
We first prove \eqref{item-tech1} by decreasing induction on $d$. For $d=r$ we have $x_{j-1} \in \Ptilde_{r,j}= P_{j-1}$ so there is nothing to prove. Fix an integer $1\leq d< r$. By Remark~\ref{StabG1} we have $x_{i} = y_{i}$ for all $i\neq j$. By induction, we may assume $y_{d+j} = x_{d+j}\in \Ptilde_{d+1,j}$. Applying \Cor~\ref{CorReduct}~\eqref{item-cor1} to $\Ptilde_{d+1,j}$ and $Q_{d+j}$ in $G_{d+j}$ (both containing the torus $T_{d+j}$), we obtain
\begin{equation} \label{equ1}
\theta_{M_{d+j}}^{Q_{d+j}}(y_{d+j}) \in \Ptilde_{d+1,j}.
\end{equation}
By definition of $E$, one has  for all $i\in \ZZ$
\begin{equation}\label{equ2}
\varphi\left(\theta_{L_{i}}^{P_{i}}(x_{i})\right) = \theta_{M_{i+1}}^{Q_{i+1}}(y_{i+1}).
\end{equation}
Combining \eqref{equ1} and \eqref{equ2} for $i=d+j-1$, we get $\varphi\left(\theta_{L_{d+j-1}}^{P_{d+j-1}}(x_{d+j-1})\right)\in \Ptilde_{d+1,j}$. It follows that $\theta_{L_{d+j-1}}^{P_{d+j-1}}(x_{d+j-1})$ lies in $P_{d+j-1}\cap\sigma^{-1}(\Ptilde_{d+1,j})=\Ptilde_{d,j}$ (because $\varepsilon$ is a $k$-valued point). One has $B^-_{d+j-1} \subset \Ptilde_{d,j} \subset P_{d+j-1}$, so $R_u(P_{d+j-1})\subset \Ptilde_{d,j}$. We deduce $x_{d+j-1}\in \Ptilde_{d,j}$, which proves the first part of \eqref{item-tech1}.

The second part of \eqref{item-tech1} is proved by increasing induction on $d$. It is clear for $d=1$. For $1<d\leq r$, we assume by induction $y_{d+j-1}=x_{d+j-1}\in \Qtilde_{d-1,j}$. By \Cor~\ref{CorReduct}~\eqref{item-cor2}, we have $\theta^{P_{d+j-1}}_{L_{d+j-1}}(x_{d+j-1}) \in \Qtilde_{d-1,j}$. Using \eqref{equ2} for $i=d+j-1$, we find  $\theta^{Q_{d+j}}_{M_{d+j}}(y_{d+j}) \in Q_{d+j} \cap \sigma(\Qtilde_{d-1,j}) = \Qtilde_{d,j}$. Since $R_u(Q_{d+j})\subset \Qtilde_{d,j}$ we deduce $y_{d+j}\in \Qtilde_{d,j}$, which shows the second part of \eqref{item-tech1}.

We now prove \eqref{item-tech2}. We must show $\varphi^{r}\left(\theta^{P'_j}_{L'_j}(x_j)\right)=\theta^{Q'_j}_{M'_j}(y_j)$. For this we use again an auxiliary parabolic subgroup $\widehat{P}_{d,j}$ in $G_j$ defined for all $j$ and $1\leq d\leq r$ by:
\begin{equation}
\widehat{P}_{d,j}:=\bigcap_{i=0}^{d-1}\sigma^{-i}(P_{i+j}).
\end{equation}
Let $\widehat{L}_{d,j}$ the Levi subgroup of $\widehat{P}_{d,j}$ containing $T_j$.
Clearly $\widehat{P}_{1,j}=P_j$ and $\widehat{P}_{r,j}=P'_j$. We use increasing induction on $1\leq d \leq r$ to show:
\[
\varphi^d \left( \theta_{\widehat{L}_{d,j}}^{\widehat{P}_{d,j}} (x_j)\right) = \theta_{\Mtilde_{d,j}}^{\Qtilde_{d,j}}(y_{d+j}).\tag{*}
\]
Note that this makes sense because of \eqref{item-tech1}. For $d=1$ this is \eqref{equ2} for $i=j$. Assume that (*) holds for $d$ and apply the operator $\varphi \circ \theta^{P_{d+j}}_{L_{d+j}}$ to (*). Using \Cor~\ref{CorReduct}~\eqref{item-cor3},\eqref{item-cor2} and the fact that $x_i=y_i$ for $i\neq j$, we get:
\begin{align*}
\varphi^{d+1} \left( \theta_{\widehat{L}_{d+1,j}}^{\widehat{P}_{d+1,j}} (x_j)\right) & = \varphi \left(\theta^{P_{d+j}}_{L_{d+j}} \left( \theta_{\Mtilde_{d,j}}^{\Qtilde_{d,j}}(y_{d+j}) \right)\right) \\
& = \varphi \left(\theta_{\Mtilde_{d,j}}^{\Qtilde_{d,j}}\left(\theta^{P_{d+j}}_{L_{d+j}} (y_{d+j}) \right)\right) \\
& =\theta_{\sigma \Mtilde_{d,j}}^{\sigma \Qtilde_{d,j}} \varphi\left(\theta^{P_{d+j}}_{L_{d+j}} (y_{d+j}) \right) \\
& =\theta_{\sigma \Mtilde_{d,j}}^{\sigma \Qtilde_{d,j}} \left(\theta^{Q_{d+j+1}}_{M_{d+j+1}} (y_{d+j+1}) \right) \\
& =\theta_{\sigma \Mtilde_{d,j} \cap M_{d+j+1}}^{\sigma \Qtilde_{d,j} \cap Q_{d+j+1}}  (y_{d+j+1}) \\
& =\theta_{\Mtilde_{d+1,j}}^{\Qtilde_{d+1,j}}(y_{d+j+1}).
\end{align*}
which is (*) for $d+1$, and this proves the claim. For $d=r$ in (*), we obtain $\varphi^{r}\left(\theta^{P'_j}_{L'_j}(x_j)\right)=\theta^{Q'_j}_{M'_j}(y_j)$, as desired.
\end{proof}

\subsection{A result on orbits}

Let $\Xcal$ denote the set of $E$-orbits in $G$, and $\Xcal_j\subset \Xcal$ the set of $E$-orbits intersecting $G_j\subset G$. By equation \eqref{Eorbs}, we have $C_{j,\alpha}\in \Xcal_j$ for all $1\leq j \leq r$ and $\alpha \in \Delta_j \setminus I_j$. Also note that $G_\eta =E\cdot 1 \in \Xcal_j$ for all $j$. The set of $E$-orbits (resp. $E'_j$-orbits) is a partially ordered set with respect to closure relations.

\begin{theorem}
\label{bijorbits}The map $C\mapsto C\cap G_{j}$ defines a
bijection between $\Xcal_{j}$ and the set of $E'_{j}$-orbits in $G_{j}$. Furthermore one has $\codim_G(C) = \codim_{G_j}(C\cap G_j)$ for all $C\in \Xcal_j$.
\end{theorem}
\begin{proof}
First we prove that $C\cap G_{j}$ is an $E'_{j}$-orbit. Let $u,v\in C\cap G_{j}$. We can find an element 
\[
\varepsilon=\left(\left(x_{1},\dots,x_{r}\right),\left(y{}_{1},\dots,y_{r}\right)\right)\in E
\]
such that $\varepsilon\cdot u=v$. Then one has $\varepsilon \in E^{(j)}$ by \Rmk~\ref{StabG1}, so $(x_{j},y_{j})\in E'_{j}$ by \Lem~\ref{technical}. Thus $\sigma\cap G_{j}$ is contained in a $E'_{j}$-orbit. Let $g\in C\cap G_{j}$ and $(x,y)\in E'_{j}$. Write $u_j:=u_j(x,y)$ and $v_j:=v_j(x,y)$ defined as in \eqref{injEprime}. One has $(u_j,v_j)\in E^{(j)}$ and $xgy^{-1}=u_jgv_j^{-1}$, so $C\cap G_{j}$ is exactly an $E'_{j}$-orbit. Hence the map $C\mapsto C\cap G_j$ is well-defined. The bijectivity is clear.

We now prove the second assertion. Let $g\in C\cap G_j$ and $\varepsilon=((x_1,...,x_r),(y_1,...,y_r))\in \Stab_E(g)(k)$. It follows from \Lem~\ref{technical}~\eqref{item-tech2} that $(x_j,y_j)\in \Stab_{E'_j}(g)$. Hence we obtain a homomorphism
\begin{equation}\label{stabmap}
\delta_j: \Stab_E(g)_{\rm red} \to \Stab_{E'_j}(g)_{\rm red}, \quad \varepsilon\mapsto (x_j,y_j).
\end{equation}
The injective homomorphism $\gamma_j:E'_j \to E$ induces a map $\Stab_{E'_j}(g) \to \Stab_E(g)$ such that $\delta_j \circ \gamma_j = \textrm{id}$, so in particular, $\delta_j$ is surjective. Note that $K_j:=\Ker(\delta_j)$, as a subgroup of $E$, is independent of $C\in \Xcal_j$ and of $g\in C\cap G_j$. Taking $g=1$, we see that $K_j$ is finite because $\Stab_E(1)$ is. The result follows.
\end{proof}

\subsection{Stabilizers}
Define subgroups $S:=\Stab_E(1)$ and $E'_j:=\Stab_{E'_j}(1)$. We just saw in \eqref{stabmap} that there is a surjective morphism $\delta_j:S_{\rm red} \to (S'_{j})_{\rm red}$ with finite kernel $K_j$. More precisely:
\begin{lemma}
One has $K_j=1$. In particular, $\delta_j$ induces an isomorphism $S_{\rm red} \to (S'_{j})_{\rm red}$.
\end{lemma}
\begin{proof}
Let $\varepsilon=(\underline{x},\underline{x})\in K_j$ with $\underline{x}=(x_1,...,x_r)$. By \Lem~\ref{ContainQ0}~\eqref{item-contain3}, one has $x_i\in L_i\cap M_i$ for all $1\leq i\leq r$. The condition $\varepsilon\in E$ then implies $\varphi(x_i)=x_{i+1}$ for all $i$. Since $\varepsilon\in K_j$, we have $x_j=1$, so the result follows.
\end{proof}

\section{Hasse invariants}
\subsection{Main result}
Let $\mu:\GG_{m,k}\to G_k$ be a cocharacter and $\Zcal=(G,P,L,Q,M,\varphi)$ the attached zip datum. Let $(B,T,z)$ be an $\FF_q$-frame for $\Zcal$, with $z=w_0 w_{0,I}$. Note that $L\subset \sigma^{-1}Q$ is a Levi subgroup, so $X^*(\sigma^{-1}Q)=X^*(L)$.
\begin{definition}\label{Zampledef}
A character $\lambda \in X^*(L)$ is $\Zcal$-ample if one of the following equivalent conditions is satisfied
\begin{equivlist}
\item The attached line bundle on $G/\sigma^{-1}Q$ is anti-ample.
\item One has $\langle \lambda, \alpha^\vee \rangle <0$ for all $\alpha\in \Delta \setminus \sigma^{-1}J$.
\end{equivlist}
\end{definition}

Let $G_1,G_2$ be connected $\FF_q$-reductive groups and $\mu_1\in X^*(G_1)$, $\mu_2\in X^*(G_2)$ cocharacters. Let $f:G_1\to G_2$ be an injective homomorphism defined over $\FF_p$. For $i=1,2$, write $\Zcal_{\mu_i}=(G_i,P_i,L_i,Q_i,M_i,\varphi)$. Recall that $f$ induces a homomorphism $L_1\to L_2$ (\S\ref{subsec-cochar}).

\begin{lemma}\label{amplepb}
If $\lambda \in X^*(L_2)$ is $\Zcal_2$-ample, then $f^*(\lambda) \in X^*(L_1)$ is $\Zcal_1$-ample.
\end{lemma}

\begin{proof}
Since $f^{-1}(Q_1)=Q_2$ and since $f$ is defined over $\FF_p$, we have an embedding $G_1/\sigma^{-1}Q_1\to G_2/\sigma^{-1}Q_2$. Hence the claim follows from (ii) of \Def~\ref{Zampledef}.
\end{proof}

\begin{definition}
Let $\lambda\in X^*(L)$ be a character. We say that a section $f\in H^0(\GZip^\mu,\Vscr(\lambda))$ is a Hasse invariant if its non-vanishing locus is exactly $U_\mu$.
\end{definition}

We now state the mail result of this article:

\begin{theorem}\label{mainthm}
If $\lambda \in X^*(L)$ is $\Zcal$-ample, there exists a Hasse invariant $h\in H^0(\GZip^\mu,\Vscr(N_\mu \lambda))$. It is unique up to scalar.
\end{theorem}

The uniqueness statement holds by \Prop~\ref{GlobalSec}. To prove the existence we proceed in several steps. We first consider the case when $P=B_-$ is a Borel, then the more general case when $P$ is defined over $\FF_q$ (in these cases, note that $G_\eta$ coincides with the open $P\times Q$ orbit in $G$). Then we show the result for a Weil restriction of degree $r$ when $P$ defined over $\FF_{q^r}$. Finally we prove the general case.

\subsection{The Borel case}
Assume first $P=B_-$, and thus $Q=B$. For $\chi\in X^*(T)$, denote by $\Lscr(\chi)$ the attached line bundle on $G/B$. If $\chi$ is $B$-dominant, define $V_\chi:=H^0(G/B,\Lscr(\chi))$. It is a $G$-representation, and it is known (\cite[\S1, p. 654]{brion-lakshmibai-monomial}) that there is a unique $B$-eigenvector $f_\chi\in V_\chi$, and the corresponding character is $-w_0\chi$. Thus $f_\chi$ identifies with a function $f_\chi:G\to \AA_1$ satisfying $f_\chi(tugt'u')=\chi(w_0 t w_0)\chi(t')f_\chi(g)$ for all $t,t'\in T$, $u,u'\in R_u(B)$ and $g\in G$. Furthermore the divisor of $f_\chi$ is
\begin{equation}
\Div(f_\chi)=\sum_{\alpha \in \Delta} \langle \chi, \alpha^\vee \rangle D_\alpha
\end{equation}
where $D_\alpha =\overline{Bw_0 s_\alpha B}$. The function $h_\chi(g):=f_\chi(w_0g)$ satisfies the relation
\begin{equation} \label{relationh}
h_\chi(agb^{-1})=\chi(\overline{a})\chi(\overline{b})^{-1} h_\chi(g) = \chi(\overline{a})\chi(\varphi(\overline{a}))^{-1} h_\chi(g)
\end{equation}
for all $(a,b)\in E$ and $g\in G$. Hence $h_\chi \in H^0(\GZip^\mu,\Vscr(\lambda))$ for $\lambda=\chi-\chi \circ \varphi$.

Now let $\lambda\in X^*(T)$ be $\Zcal$-ample and fix $m\geq 1$ such that $\lambda$ is defined over $\FF_{q^m}$. One has $\lambda =\chi-\chi \circ \varphi$ for
\begin{equation}\label{chiform}
\chi=-\frac{1}{p^m-1} \left(\lambda + \lambda\circ \varphi +...+ \lambda \circ \varphi^{m-1}\right) \in X^*(T) \otimes \QQ.
\end{equation}
Since $\lambda$ is $\Zcal$-ample, the character $\lambda \circ \varphi^r$ is $\Zcal$-ample for all $r\geq 1$, hence $(p^m-1)\chi$ is $B$-dominant regular character. It follows from the previous discussion that $\Vscr((p^m-1)\lambda)$ admits a Hasse invariant $h_1$. Let $h_2\in H^0(U_\mu,\Vscr(N_\mu \lambda))$ given by \Prop~\ref{GlobalSec}~\eqref{item-Glob2}. Then necessarily $h_1^{N_\mu}=h_2^{p^m-1}$ (up to a nonzero scalar) by the same lemma. Since $G$ is normal, it follows that $h_2$ extends to a Hasse invariant.

\subsection{The case when $P$ defined over $\FF_q$} \label{subsecPratcase}
Assume $P$ is defined over $\FF_q$. Then so is $Q$, and $M=L=P\cap Q$. Since $G_\eta$ is a $P\times Q$-orbit, it is a union of $B_-\times B$-orbits. Its complement is 
\begin{equation}
G\setminus G_\eta =\bigcup_{\alpha \in \Delta \setminus J} w_0 C_\alpha.
\end{equation}
Define $\chi \in X^*(L)\otimes \QQ$ as in \eqref{chiform}, and $\chi_2:=(p^m-1)\chi \in X^*(L)$, which is $B$-dominant. By the previous case, we can find a regular function $h_{\chi_2}:G\to\AA_1$ whose divisor is 
\begin{equation}
\Div(h_{\chi_2})=\sum_{\alpha \in \Delta} \langle \chi_2, \alpha^\vee \rangle w_0 D_\alpha=\sum_{\alpha \in \Delta\setminus J} \langle \chi_2, \alpha^\vee \rangle w_0 D_\alpha.
\end{equation}
Since $\lambda$ is $\Zcal$-ample, $-\chi_2$ is $\Zcal$-ample, hence $ \langle \chi_2, \alpha^\vee \rangle>0$ for all $\alpha \in \Delta\setminus J$. We claim that $h_{\chi_2}$ is a section of $\Vscr((p^m-1)\lambda)$. Since $\Div(h_{\chi_2})$ is $E$-equivariant, it follows from \cite[\S1]{Knop-Kraft-Vust-G-variety} that $h_{\chi_2}$ is an $E$-eigenfunction, i.e there is a character $\theta \in X^*(L)$ such that $h_{\chi_2}(agb^{-1})=\theta(\overline{a}) h_{\chi_2}(g)$ for all $(a,b)\in E$ and $g\in G$. Taking $a,b\in T$ and comparing with \eqref{relationh} shows $\theta =\chi_2-\chi_2 \circ \varphi=(p^m-1)\lambda$, which proves the claim. Hence $h_{\chi_2}\in H^0(\GZip^\mu,\Vscr((p^m-1)\lambda)$ is a Hasse invariant. Finally, we conclude as in the Borel case that $\Vscr(N_\mu\lambda)$ admits a Hasse invariant too.

\subsection{The Weil restriction case}\label{subsecWeilcase}
Let $G=\Res_{\FF_{q^r}/\FF_q}(G_1)$ as in \S\ref{Weilres} and retain the notations therein. For $\lambda \in X^*(L)$, denote by $\lambda_j\in X^*(L_j)$ the character induced by the inclusion $L_j\to L$. Let $\lambda=(\lambda_1,...,\lambda_r) \in X^*(L)$ and $f\in H^0(U_\Zcal,\Vscr(\lambda))$ nonzero (hence non-vanishing). Since $G_\eta \cap G_j$ is the open $E'_j$-orbit in $G_j$ (\Lem~\ref{bijorbits}), $f$ restricts to a non-vanishing section $f_j$ over $U_{\Zcal_j}$ of a certain line bundle, determined in the next proposition.

\begin{proposition}\label{propHasseWeil}
Let $\lambda \in X^*(L)$ and $f\in H^0(U_\Zcal,\Vscr(\lambda))$.
\begin{assertionlist}
\item \label{itemsection1} One has $f_j \in  H^0(U_{\Zcal_j},\Vscr(\lambda'_j))$, for $\lambda'_j:=\sum_{i=0}^{r-1}\lambda_{i+j}\circ \varphi^{i}\in X^*(L'_j)$ (indices taken modulo $r$).
\item \label{itemsection2} $f$ extends to $\GZip^\Zcal$ if and only if $f_j$ extends to $\GjZ^{\Zcal_j}$ for all $1\leq j \leq r$.
\item \label{itemsection3} $f$ extends to a Hasse invariant if and only if $f_j$ does for all $1\leq j \leq r$.
\end{assertionlist}
\end{proposition}

\begin{proof}
We prove \eqref{itemsection1}. Consider the map $\gamma_j:E'_j\to E$ defined in \eqref{injEprime}. For all $\varepsilon_j=(a_j,b_j)\in E'_j$ and $g_j\in G_j$, note that $\varepsilon_j\cdot g_j=\gamma_j(\varepsilon)\cdot g_j$. Hence $f_j$ satisfies $f_j(\varepsilon_j\cdot g_j)=f(\gamma_j(\varepsilon)\cdot g_j)=\lambda(\gamma_j(\varepsilon))f(g_j)=\lambda(\gamma_j(\varepsilon))f_j(g_j)$. Since $\lambda'_j=\lambda\circ \gamma_j$, the result is proved.

We can write $\Div(f)=\sum_{j=1}^r \sum_{\alpha\in \Delta_j\setminus I_j} n_{j,\alpha}\overline{C}_{j,\alpha}$ with $n_{j,\alpha}\in \ZZ$, using the notation introduced in \eqref{Eorbs}. Since $C':=C_{j,\alpha}\cap G_j$ is a codimension one $E'_j$-orbit, the sign of $n_{j,\alpha}$ is the same as the sign of the multiplicity of $f_j$ along $\overline{C'}$. This proves assertions \eqref{itemsection2} and \eqref{itemsection3}.
\end{proof}

\begin{lemma}\label{amplerest}
Let $\lambda\in X^*(L)$ be a $\Zcal$-ample character. Then $\lambda'_j$ is $\Zcal_j$-ample.
\end{lemma}

\begin{proof}
Let $\alpha \in \Delta_j\setminus I_j$. One has $
\langle \lambda'_j,\alpha^\vee \rangle=\sum_{i=0}^{r-1}\langle\lambda_{i+j}\circ \varphi^{i}, \alpha^\vee \rangle=\sum_{i=0}^{r-1}p^i\langle\lambda_{i+j} , \sigma^i(\alpha^\vee) \rangle.$
All the summands are non-positive numbers and $\langle\lambda_{j} ,\alpha^\vee \rangle<0$, hence the result.
\end{proof}

\begin{corollary}\label{corWeilpartial}
Assume that $P_1,...,P_r$ are defined over $\FF_{q^r}$. If $\lambda\in X^*(L)$ is $\Zcal$-ample, there exists a Hasse invariant $h\in H^0(\GZip^\mu,\Vscr(N_\mu \lambda))$.
\end{corollary}

\begin{proof}
Let $h\in H^0(U_\Zcal,\Vscr(N_\mu \lambda))$ denote the section of \Prop~\ref{GlobalSec}~\eqref{item-Glob2}. Using \Prop~\ref{propHasseWeil}, we need to show that $h_j\in H^0(U_{\Zcal_j},\Vscr(\lambda'_j))$ extends to $\GjZ^{\Zcal_j}$. By \Lem~\ref{amplerest}, $\lambda'_j\in X^*(L'_j)$ is $\Zcal_j$-ample. Since $\Zcal_j$ is a zip datum over $\FF_{q^r}$, and $P'_j$ is defined over $\FF_{q^r}$ by assumption, the result follows from the previous case (\S\ref{subsecPratcase}).
\end{proof}

\subsection{The general case}
Let $G$ be a connected reductive $\FF_q$-group, $\mu:\GG_{m,k}\to G_k$ a cocharacter, and $\Zcal=(G,P,L,Q,M,\varphi)$ its zip datum over $\FF_q$. Assume $(B,T,z)$ is an $\FF_q$-frame with $z=w_0 w_{0,I}$. Fix $r\geq 1$ such that $P$ is defined over $\FF_{q^r}$. Define $\tilde{G}:=\Res_{\FF_{q^r}/\FF_q}(G)$ and identify $G_k$ with $G_k\times...\times G_k$. Consider the diagonal embedding $\iota:G\to \tilde{G}$. The cocharacter $\tilde{\mu}:=\iota \circ \mu$ defines a zip datum $\tilde{\Zcal}=(\tilde{G},\tilde{P},\tilde{L},\tilde{Q},\tilde{M},\varphi)$, where $\tilde{P}=P\times...\times P$ and similarly for $L,Q,M$.

Similarly we define $\tilde{T}:=T\times...\times T$, $\tilde{B}:=B\times...\times B$ and $\tilde{z}:=(z,...,z)$. Then $(\tilde{B},\tilde{T},\tilde{z})$ is an $\FF_q$-frame for $\tilde{\Zcal}$. The embedding $\iota$ induces naturally a morphism of stacks $\iota^{\#}:\GZip^\mu \to \GtildeZ^{\tilde{\mu}}$. Zip strata in $\tilde{G}$ are parametrized by the set ${}^{\tilde{I}}\tilde{W}:={}^I W\times ... \times {}^I W$, where $I$ is the type of $P$. Using \eqref{param}, it is clear that $\iota^{\#}$ induces the diagonal embedding ${}^I W \to {}^{\tilde{I}}\tilde{W}$ on zip strata. In particular:

\begin{lemma}\label{opentilde}
One has $(\iota^{\#})^{-1}(U_{\tilde{\mu}})=U_\mu$.
\end{lemma}

If $\tilde{\lambda}=(\lambda_1,...,\lambda_r)\in X^*(\Ltilde)$, then one has $(\iota^{\#})^{*}(\Vscr(\tilde{\lambda}))=\Vscr(\lambda_1+...+\lambda_r)$. Hence \Lem~\ref{opentilde} implies that the restriction of a Hasse invariant for the weight $\tilde{\lambda}$ is a Hasse invariant for the weight $\lambda_1+...+\lambda_r$.

Let $\lambda \in X^*(L)$ be a $\Zcal$-ample character and set $\tilde{\lambda}:=(\lambda,...,\lambda)$. It is clearly $\tilde{\Zcal}$-ample. Since $\tilde{P}$ is defined over $\FF_{q^r}$, \Cor~\ref{corWeilpartial} implies that there exists a Hasse invariant $\tilde{h}$ for $\Vscr(N_{\tilde{\mu}}\tilde{\lambda}))$. Its restriction to $\GZip^\mu$ is a Hasse invariant for the weight $rN_{\tilde{\mu}}\lambda$. As before, we conclude that there is also a Hasse invariant for the weight $N_\mu \lambda$. This terminates the proof of \Th~\ref{mainthm}.

\subsection{Shimura varieties of Hodge-type}
Let $(\mathbf{G},X)$ be a Shimura datum of Hodge type with reflex field $E$, given by a reductive group $\mathbf{G}$ over $\QQ$ and a $\mathbf{G}(\RR)$-conjugacy class of morphisms $\SS\to \mathbf{G}_\RR$. For $h_0\in X$, define $\mu_0:\mathbf{G}_{m,\CC}\to \mathbf{G}_{\CC}$ as $\mu_0(z)=h_{0,\CC}(z,1)$ using the identification $\SS_\CC\simeq \GG_{m,\CC}\times \GG_{m,\CC}$ given by $z\mapsto (z,\overline{z})$ on $\RR$-points.

Fix a prime number $p$ such that $\mathbb{G}_{\QQ_p}$ is unramified and let $\Gcal$ be a reductive $\ZZ_p$-model of $\mathbb{G}_{\QQ_p}$. Fix a neat compact open subgroup $K\subset \mathbf{G}(\AA_f)$ such that $K=K^p K_p$ with $K^p \subset \mathbf{G}(\AA_f^p)$ and $K_p=\Gcal(\ZZ_p)$. Let $Sh_K(\mathbf{G},X)$ denote the attached Shimura variety over $E$. For a prime $\pfr$ of $E$ above $p$, there is a smooth canonical $\Ocal_{E,\pfr}$-model $\Scal_K$ of $Sh_K(\mathbf{G},X)$, by \cite{Vasiu-Preabelian-integral-canonical-models} and \cite{Kisin-Hodge-Type-Shimura}.

Let $S_K:=\Scal_K\otimes_{\Ocal_{E,\pfr}}\kappa(\pfr)$ be its special fiber, where $\kappa(\pfr)$ is the residual field of $\pfr$. The conjugacy class of $\mu_0$ can be lifted to $\Gcal$, and there exists a representative $\mu$ defined over $\Ocal_{E_\pfr}$ because $\Gcal$ is quasi-split. Define $G:=\Gcal \otimes_{\ZZ_p}\FF_p$ and write again $\mu\in X_*(G)$ for the induced cocharacter.

Let $f:\Ascr \to \Scal_K$ be the universal abelian scheme obtained by pull-back from the Siegel-type Shimura variety. The de Rham cohomology $H^1_{dR}(A / S_K)$ together with its Hodge filtration and conjugate filtration define naturally a
$G$-zip over $S_K$ (\cite[\Th~2.4.1]{ZhangEOHodge}). This induces a smooth morphism of stacks (\cite[\Th~3.1.2]{ZhangEOHodge})
\begin{equation}
\zeta:S_K \longrightarrow \GZip^\mu.
\end{equation}
The geometric fibers of $\zeta$ are the Ekedahl-Oort strata of $S_K$. In particular, the open stratum $S_{K,\mu}:=\zeta^{-1}(U_\mu)$ is called the $\mu$-ordinary locus of $S_K$. The Hodge line bundle of $\Scal_K$ is defined as
\begin{equation}
\omega:=f_*(\det(\Omega_{\Ascr/\Scal_K})).
\end{equation}
It is ample on $\Scal_K$ by \cite{Moret-Bailly-abelian-varieties-book}. Denote by $\mathbf{P}\subset \mathbf{G}_{\overline{\QQ}}$ the parabolic subgroup attached to $(\mathbf{G},\mu_0)$. For a character $\lambda \in X^*(\mathbf{P})$, there is an automorphic line bundle $\Vcal_K(\lambda)$ on $\Scal_K$ attached to $\lambda$ and its special fiber coincides with the line bundle $\zeta^*(\Vcal(\lambda))$. Furthermore, there exists a character $\lambda_\omega\in X^*(\mathbf{P})$ such that $\omega=\Vcal(\lambda_\omega)$. Denote by $N_\mu$ the integer defined in \S\ref{subsecline}.

\begin{corollary}\label{mainShimura}
There exists a section $h_K\in H^0(S_K,\omega^{N_\mu})$ whose non-vanishing locus is $S_{K,\mu}$.
\end{corollary}

\begin{proof}
It suffices to check that $\lambda_\omega$ is $\Zcal$-ample. This is clear for Siegel-type Shimura varieties. The general Hodge-type case then follows from \Lem~\ref{amplepb}.
\end{proof}

\begin{rmk}
The section $h_K$ coincides with the one constructed by W. Goldring and the first author as a special case of \cite[\Cor~4.2.3]{Goldring-Koskivirta-Strata-Hasse} since the $k$-vector space $H^0(\GZip^\mu,\Vscr(\lambda))$ has at most dimension 1 for all $\lambda\in X^*(L)$ (Proposition \ref{GlobalSec}\eqref{item-Glob2}).
\end{rmk}

Let $\Scal_K^{\rm min}$ denote the minimal compactification of $\Scal_K$ constructed by Madapusi Pera and let $S_K^{\rm min}$ denote its special fiber. By \loccitn, the Hodge bundle $\omega$ extends naturally to an ample line bundle on $\Scal_K^{\rm min}$, which we continue to denote by $\omega$. We make the assumption that $\mathbf{G}^{\rm ad}$ has no factor isomorphic to $PGL_{2,\QQ}$. In this case, the boundary of $S_K^{\rm min}$ has codimension $>2$. By normality of $S_K^{\rm min}$, the section $h_K$ extends uniquely to a section of $\omega^{N_\mu}$ over $S^{\rm min}_K$. Define the $\mu$-ordinary locus $S^{\rm min}_{K,\mu}$ as the non-vanishing locus of this extension. Since $S_K^{\rm min}$ is projective, we deduce:

\begin{corollary}\label{muordaffine}
The $\mu$-ordinary locus $S^{\rm min}_{K,\mu}$ is affine.
\end{corollary}

\subsection{Calculation of $N_{\mu}$ for Shimura varieties of PEL type}\label{CalcNmu}

We now calculate the integer $N_{\mu}$ for pairs $(G,\mu)$ defined by a Shimura datum of PEL type with $G$ connected. Retain the notation of \S\ref{stabilizer}. Recall that $N_\mu$ is the exponent of $X^*(S)=\Hom(L_0(\FF_p),k^\times)$ (\Def~\ref{defNeta}). The classification of PEL data (see e.g., \cite[2.2 and 2.3]{Wedhorn-ordinariness-Shimura-varieties}) shows that there is an exact sequence
\[
1 \to \prod_{i=1}^m G_i \to G \xrightarrow{\eta}\GG_{m,\FF_p} \to 1.
\]
where each group $G_i$ is one of the following reductive groups.
\begin{enumerate}
\item[(AL)] $\Res_{\FF_{p^r}/\FF_p}GL_n$,
\item[(AU)] $\Res_{\FF_{p^r}/\FF_p}U(n)$,
\item[(C)] $\Res_{\FF_{p^r}/\FF_p}Sp_{2g}$.
\end{enumerate}
Concretely one often has $m = 1$, which we will assume for simplicity. Hence we have an exact sequence
\begin{equation}\label{EqMultiplierSeq}
1 \to G_1 \to G \xrightarrow{\eta}\GG_{m,\FF_p} \to 1,
\end{equation}
where $G_1=\Res_{\FF_{p^r}/\FF_p}G'$ for $G'\in\{GL_n,U(n),Sp_{2g}\}$. Any Levi $\FF_p$-subgroup of $G_1$ is of the form $\Res_{\FF_{p^r}/\FF_p}L'$ for some Levi $\FF_{p^r}$-subgroup $L'\subset G'$. In particular, there exists a Levi $\FF_{p^r}$-subgroup $L'_0$ of $G'$ such that $L_0 \cap G_1 = \Res_{\FF_{p^r}/\FF_p}L_0'$. As finite fields are of cohomological dimension $1$ we have an exact sequence
\begin{equation}\label{EqLeviSeq}
1 \to L'_0(\FF_{p^r}) \to L_0(\FF_p) \to \FF_p^{\times} \to 1.
\end{equation}

\textbf{The case $G' = GL_n$}. In this case the sequences \eqref{EqMultiplierSeq} and \eqref{EqLeviSeq} split and $L'_0$ is a product of general linear groups. Hence the abelianization $L'(\FF_{p^r})^{\rm ab}$ of $L'(\FF_{p^r})$ is a product of copies of $\FF^{\times}_{p^r}$ except if one of the general linear groups is $GL_2$, $p = 2$, and $r = 1$ ($GL_2(\FF_2)$ is isomorphic to the symmetric group $S_3$ and $GL_2(\FF_2)^{\rm ab}$ has $2$ elements). The exponent of $\Hom(L'_0(\FF_{p^r}),k^{\times})$ is the same as that of the prime-to-$p$-part of $L'(\FF_{p^r})^{\rm ab}$, which is $p^r-1$ in any case. As \eqref{EqLeviSeq} splits, $N_{\mu}=\lcm(p^r - 1,p-1)= p^r - 1$.

\textbf{The case $G' = U(n)$}. In this case $L_0$ is isomorphic to a product of groups of the form $H:=\Res_{\FF_{p^{2r}}/\FF_p}(GL_d)$ and at most one group $M$ whose $\FF_p$-valued points are given by
$$
M(\FF_p) = \{g \in GU(V,\phi), \eta(g) \in \FF^\times_p\},
$$
where $(V,\phi)$ is an hermitian space over $\FF_{p^r}$. By the first case, the exponent of $X^*(H(\FF_p))$ is $p^{2r}-1$ for each such group $H$. Over $k$, one has $M_k\simeq GL_m\times...\times GL_m \times \GG_m$. Hence $X^*(M)$ identifies naturally with $\ZZ^{r+1}$ such that $(a_1,...,a_n,b)\in \ZZ^{r+1}$ corresponds to the character $(A_1,...,A_r,\lambda)\mapsto \lambda^b\prod_i \det(A_i)^{a_i}$. The Galois action on $X^*(M)$ is given by $\sigma (a_1,...,a_r,b)=(-a_r,-a_1,...,-a_{r-1},b)$. It follows from \Cor~\ref{surject} that the exponent of $X^*(M(\FF_p))$ is $p^r-(-1)^r$.  Hence $N_\mu=\lcm(p^r-(-1)^r,p^{2r}-1)=p^{2r}-1$.

\textbf{The case $G' = Sp_{2g}$}. We first claim that again \eqref{EqLeviSeq} splits (although \eqref{EqMultiplierSeq} usually does not split in this case). Let $V$ be a symplectic space of dimension $2g$ over $\FF_{p^r}$. Then \eqref{EqMultiplierSeq} is on $\FF_p$-valued points the exact sequence
\[
1 \to Sp(V) \to \{g \in GSp(V), \eta(g) \in \FF^{\times}_p \} \to \FF_p^{\times} \to 1,
\]
where $\eta$ is the multiplier homomorphism.
There is only one proper parabolic subgroup given by a minuscule cocharacter for group of Dynkin type $C$, namely the Siegel parabolic. Hence $L_0$ is the subgroup of symplectic similitudes $V \to V$ that preserve a decomposition $V = U_1 \oplus U_2$, where $U_1$ and $U_2$ are totally isotropic subspaces of dimension $g$, hence
\[
L_0(\FF_p) = \{(g_1,g_2) \in GL(U_1)(\FF_{p^r}) \times GL(U_2)(\FF_{p^r}), g_1g_2^{-1*} \in \FF_p^{\times}\id_{U_1}\}
\]
where $(\ )^*$ denotes the duality of $U_1$ and $U_2$ induced by to the symplectic pairing. One has $GL(U_1)(\FF_{p^r}) \cong L'_0(\FF_{p})$ by $g_1 \mapsto (g_1,g_1^{-1*})$. Therefore a splitting of \eqref{EqLeviSeq} is given by $\FF_p^{\times} \to L_0(\FF_p)$, $\alpha \mapsto (\alpha\id_{U_1},\id_{U_2})$. 

As above, we deduce $N_{\mu} = p^r-1$.

\begin{rmk}
In the unitary case, the integer $N_\mu=p^{2r}-1$ agrees with the results of \cite{Goldring-Nicole-mu-Hasse}. One can show that the Hasse invariant of \Cor~\ref{mainShimura} coincides with those of \cite{Goldring-Nicole-mu-Hasse} and \cite{Boxer-thesis}.
\end{rmk}

\bibliographystyle{amsalpha}
\bibliography{bibliography_2016}

\end{document}